\documentclass[journal]{IEEEtran}
\usepackage{amsmath}
\usepackage{amssymb}
\usepackage{amsthm}
\usepackage{arydshln}
\usepackage{nicematrix}
\usepackage{array}
\usepackage{cite}
\usepackage{nameref,hyperref}
\usepackage{caption}
\usepackage{subcaption}
\usepackage[shortlabels]{enumitem}
\usepackage[utf8]{inputenc}
\usepackage{color,soul}
\usepackage[scr]{rsfso}
\usepackage{nicematrix}
\usepackage{arydshln}
\usepackage{comment}
\usepackage[safe]{tipa}
\newcommand{\Laplace}{\mathscr{L}}
\usepackage{xcolor}

\usepackage{nccmath}
\usepackage[strict]{changepage}

\usepackage{lipsum}

\newtheorem{definition}{Definition}
\newtheorem{theorem}{Theorem}

\theoremstyle{remark}
\newtheorem{assumption}{Assumption}
\newtheorem{remark}{Remark}
\newtheorem{problem}{Problem}
\newtheorem{corollary}{Corollary}
\newtheorem{example}{Example}

\DeclareMathOperator*{\rank}{rank}

\DeclareMathOperator*{\im}{Im}
\DeclareMathOperator*{\proj}{proj}
\DeclareMathOperator*{\res}{res}

\DeclareMathOperator*{\minimize}{minimize}
\DeclareMathOperator*{\subj2}{subject\ to}

\expandafter\def\expandafter\normalsize\expandafter{%
    \normalsize%
    \setlength\abovedisplayskip{4pt}%
    \setlength\belowdisplayskip{4pt}%
    \setlength\abovedisplayshortskip{0pt}%
    \setlength\belowdisplayshortskip{0pt}%
}

\allowdisplaybreaks

\usepackage{array}
\usepackage{hyperref}   
\hypersetup{
	colorlinks=true,
	linkcolor=blue,
	citecolor=blue,
	urlcolor=blue,
}

\begin{document}

\title{Algebraic Control: Complete Stable Inversion with Necessary and Sufficient Conditions}

\author{Burak~K\"{u}rk\c{c}\"{u}, \textit{Member}, IEEE, Masayoshi Tomizuka, \textit{Life Fellow}, IEEE%
\thanks{Burak~K\"{u}rk\c{c}\"{u} is with the Department of Electrical and Computer Engineering,
        Santa Clara University, Santa Clara, CA 95053, USA
        (e-mail: bkurkcu@scu.edu)}%
\thanks{ Masayoshi Tomizuka is with the Department of
Mechanical Engineering, University of California, Berkeley,
CA 94720, USA (e-mail: tomizuka@berkeley.edu).}

}

\markboth{}%
{Shell \MakeLowercase{\textit{et al.}}: Bare Demo of IEEEtran.cls for IEEE Journals}

\maketitle

\begin{abstract}
In this paper, we establish necessary and sufficient conditions for stable inversion, addressing challenges in non-minimum phase, non-square, and singular systems. An $\mathcal{H}_\infty$-based algebraic approximation is introduced for near-perfect tracking without preview. Additionally, we propose a novel robust control strategy combining the nominal model with dual feedforward control to form a feedback structure. Numerical comparison demonstrates the approach's effectiveness.
\end{abstract}


\IEEEpeerreviewmaketitle

\section{Introduction} Recent advancements in autonomous systems and robotics have increased interest in stable inversion to meet performance demands. Learning-based control methods have further broadened research into inversion \cite{de2022frequency}, emphasizing the need to re-examine the conditions for stable inversion.

The study of stable inversion began with Brockett's introduction of \textit{functional reproducibility} in 1965 \cite{brockett1965reproducibility}. Silverman extended these concepts to multivariable systems in 1969 \cite{silverman1969inversion}, followed by geometric formulations introduced by Basile and Marro \cite{basile1973new} and applications in reduced-order control by Moylan \cite{moylan1977stable}. The stable inversion in discrete-time systems was first addressed by Tomizuka \cite{tomizuka1987zero}, while Hunt explored non-causal inversion techniques \cite{hunt1996noncausal}. 

Recent advancements include stable inversion for SISO affine systems \cite{spiegel2024stable}, square systems in continuous time \cite{saeki2023model} and discrete time \cite{naderi2018inversion}, and approximate inversion with preview\cite{romagnoli2019general}. Additionally, {strong inversion} for handling initial states in multivariable systems \cite{Loreto2023} and geometric methods for convolution-based inversion \cite{marro2002convolution} have further expanded the field.

As the field evolves, the need for system classification and corresponding solutions becomes evident. The primary objective of inversion-based approaches is to determine a bounded input that accurately reproduces the desired or given output. However, this goal poses challenges, particularly with non-minimum phase systems, which complicate the establishment of stable inputs under causality without infinite pre-actuation. Furthermore, stable inversion can impact tracking performance and stability in the presence of uncertainties \cite{devasia2002should}, where merging learning-based control with stable inversion partially addresses these issues \cite{spiegel2024stable,kim2024asymptotic}.

Despite the theoretical challenges, the practical side is equally important. Stable inversion methods have been applied across various domains, including iterative learning control \cite{van2018inversion}, trajectory tracking for autonomous systems \cite{drucker2023trajectory}, and optimal channel equalization \cite{gu2003worst}.

In this paper, we extend the concept of \textit{Zero Phase Error Tracking (ZPET)} \cite{tomizuka1987zero} to continuous-time multivariable systems. In continuous time, we may apply ZPET for compensating the phase shifts introduced by unstable zeros, but ZPET compensation alone does not make
sense due to the  high-pass nature of these zeros. However, if the overall compensator is designed to suppress these high-pass effects, ZPET remains effective at low frequencies. This insight motivates three contributions:

\noindent 1) We revisit the stable inversion problem using a similar algebraic setting of Model Matching \cite{wolovich2012linear}, yet specifying output structures leading to necessary and sufficient conditions, even for non-minimum phase, non-square, and singular systems.

\noindent 2) We propose an algebraic approximation that repurposes certain functions in $\mathcal{H}_\infty$ theory to achieve near-perfect tracking when perfect tracking conditions are not met, eliminating the need for preview or pre-actuation.

\noindent 3) We guarantee robustness without learning-based mechanisms. When limitations or uncertainties make the output set only partially reachable, we prove that the tracking error converges to an inevitable yet bounded residual.

Our goal is to develop a unified, \textit{causal} framework for stable inversion in non-minimum phase, multivariable, and uncertain systems, addressing the inherent challenges.

We organize this paper as follows: Section~\ref{sec:Prel} provides the preliminaries. Section~\ref{sec:Main} presents the main results. Section~\ref{sec:examples} illustrates numerical examples. Section~\ref{sec:conc} concludes the paper.
\section{Preliminaries} \label{sec:Prel}
\subsection{Algebraic Preliminaries} \label{sec:algebraic background}
Let \( \mathbb{R} \) be a \textit{field} of real numbers. Consider the set of all polynomials in \( s \) with coefficients in \( \mathbb{R} \). This set forms a commutative ring over \( \mathbb{R} \) and is denoted by \( \mathbb{R}[s] \). If this ring, say \( \mathcal{R} \), has an identity element and no zero divisors, then \( \mathcal{R} \) is an \textit{integral domain}. Note that \( \mathbb{R}[s] \) is also a Euclidean Domain. 
Moreover, the set of all such rational functions in \( s \) over \( \mathbb{R} \) forms a \textit{field}, denoted by \( \mathbb{R}(s) \) \cite{forney1975minimal}, and \( \mathbb{R}[s] \subset \mathbb{R}(s) \) holds. The sets of \( n_y \times n_u \) matrices with elements in \( \mathbb{R} \), \( \mathbb{R}[s] \), and \( \mathbb{R}(s) \) are denoted by \( \mathbb{R}^{n_y \times n_u} \), \( \mathbb{R}[s]^{n_y \times n_u} \), and \( \mathbb{R}(s)^{n_y \times n_u} \) respectively.  Any $n_y \times n_u$ matrix, say $ \grave{P}(s) \in \mathbb{R}[s]^{n_y \times n_u}$, over $\mathcal{R}$ 
 can be factorized (see \textit{Invariant Factor Theorem} \cite[Theorem 2.1]{kuijper2012first}) as
 \begin{align} \label{eq:invariantfactor}
 \hspace{-0.2cm}
   \grave{P}(s)&=U_R(s)S_m(s)V_R(s) \nonumber \\   
   &=[{U}_{R1} \ {U}_{R2}]  \begin{bmatrix}
    \Lambda(s) & \textbf{0}_{r \times (n_u-r)}\\
    \textbf{0}_{(n_y-r)\times r} & \textbf{0}_{(n_y-r)\times (n_u-r)}
\end{bmatrix}{\begin{bmatrix}
    {V}_{R1} \\
    {V}_{R2}
\end{bmatrix}}
 \end{align}
where $U_R(s) \in \mathbb{R}[s]^{n_y \times n_y}$ and $V_R(s) \in \mathbb{R}[s]^{n_u \times n_u}$ are $\mathcal{R}$-Unimodular matrices, $\Lambda(s)=\text{diag}[d_1(s) \ \dots \ d_r(s)]$ with unique \textit{monic} $d_i(s) \in \mathcal{R}$ such that $d_i$ divides $d_{i+1}$. The matrix $S_m(s) \in \mathbb{R}[s]^{n_y \times n_u}$ is called the \textit{Smith Normal} form of $\grave{P}$.

Let $V$ be a vector space over the field $\mathbb{R}(s)$ with dimension $k$, consisting of $n$-tuples such that a basis of vector polynomials can always be found for \( V \) [page 22, \cite{amparan2021minimal}]. A \textit{minimal} basis of $V$ is defined as a $k \times n$ polynomial matrix $P_m$ \cite{forney1975minimal}. \textit{Adapted Forney's Theorem} \cite[Section 3(4.)]{forney1975minimal} states that if $y = xP_m$ is a polynomial $n$-tuple, then $x$ must be a polynomial $k$-tuple.

For the \textit{rank} notation, akin to \cite{antoniou2005numerical}, the rank of $P$ is defined as the maximum size of any linearly independent subset of its columns in the \textit{field} \(\mathbb{R}(s)\), denoted by \(\text{rank}_{\mathbb{R}(s)}(P)\), where \(\text{rank}_{\mathbb{R}(s)}(P) \ne \text{rank}_{\mathbb{R}}(P) \). 

Suppose $\text{rank}_{\mathbb{R}(s)}(P)=r$, where $1 \leq r \leq \min(n_u,n_y)$. {Define $J=\{j_1,\dots,j_r\} \subseteq \{1,\dots,n_u\}$ as an ordered indexed set corresponding to $P(s)$'s linearly independent columns. We then define the matrix $L$ as 
\begin{equation} \label{eq:Lind_columns_P}
{L} = [ p_{j_i} \ \dots \ p_{j_r} ]\in \mathbb{R}(s)^{n_y \times r} \, 
\end{equation}
where $p_{j_k}$ denotes the $j_k$-th column of $P$. By construction, $L$ has full column rank $r$, forming a basis for the column space of $P$ while preserving its span with a minimal set of linearly independent columns.} The image of $P$ over $\mathbb{R}(s)$ is then:
\small
\begin{align}\label{eq:spanL(P)}
\text{Im}_{\mathbb{R}(s)}(P) = \left\{ \sum_{i=1}^{r} c_i p_i : c_i \in \mathbb{R}(s), p_i \in {L} \right\} \subseteq \mathbb{R}(s)^{n_y}.
\end{align}
\normalsize
Some further notations throughout the paper are: $\Vert (\cdot)(t)\Vert$ denotes the  Euclidean norm, $\Vert (\cdot) (t) \Vert_\infty\triangleq \text{ess sup}_{t \ge 0} \Vert (\cdot)(t)\Vert$, any complex number can be expressed as $\Re(\cdot)\pm j\Im(\cdot) \in \mathbb{C}$, $\sigma$ represents singular values. Time domain square-integrable functions are denoted by $\mathcal{L}_2(-\infty, \infty)$. Its causal subset is given by $\mathcal{L}_2[0, \infty)$. For the frequency domain, including at $\infty$, $\mathcal{L}_2(j\mathbb{R})$ represents square-integrable functions on $j\mathbb{R}$, and $\mathcal{L}_\infty(j\mathbb{R})$ denotes bounded functions on $\Re(s) = 0$. All these functions in Lebesgue spaces may be either matrix-valued or scalar.
Then, $\mathcal{R}\mathcal{H}_\infty$ denotes the set of real rational $\mathcal{L}_\infty(j\mathbb{R})$ functions analytic in $\Re(s) > 0$, $(\cdot)^\dagger$ represents pseudo-inverse yielding identity matrix under multiplication, $(\cdot)^*$ denotes complex conjugates transpose, and $\langle . , . \rangle$ denotes the inner product. Among various definitions available for multivariable zeros of a Transfer Function Matrix (TFM), $P(s)$, we adopt the definition given in \cite[chapter 4.5.3]{skogestad2005multivariable} as
\begin{align} \label{eq:zerosP}
\text{Z}_P &\triangleq \{z \in \mathbb{C} : P(z)u_z=0y_z\} \\
\text{Z}_R &\triangleq  \{z \in \text{Z}_P : \Re(z) > 0\} \subseteq \text{Z}_P  \label{eq:rhpzeros}
\end{align}
where $\text{Z}_R$ defines the RHP zeros making the system non-minimum phase, $u_z,y_z$ (can be obtained via singular value decomposition (SVD) of $P(z)=U\Sigma V^*$) are normalized input and output zero directions respectively.
\subsection{Problem Statement} \label{sec:problem}
Consider the general representation of a MIMO linear time-invariant (LTI) system, $\Xi: \mathbb{U} \times \mathbb{X} \to \mathbb{Y}$, as
\begin{align} \label{eq:ss0}
\Xi(u,x_0) : \begin{cases}
\dot{x}(t)&= Ax(t)+Bu(t)\\ 
 y(t)&= C x(t), \ x_0\triangleq x(0) \ne 0
\end{cases}
\end{align} 
where $A \in \mathbb{R}^{n\times n}$, $B\in \mathbb{R}^{n\times n_u}$,  $C\in \mathbb{R}^{n_y \times n}$, $x, x_0 \in \mathbb{X} \subseteq \mathbb{R}^{n}, x(t) \in \mathcal{L}_2[0,\infty),  y \in \mathbb{Y} \subset  \mathbb{R}^{n_y}, y(t) \in \mathcal{L}_2[0,\infty), u \in \mathbb{U} \subset \mathbb{R}^{n_u}$, and $u(t) \in \mathcal{L}_2[0,\infty)$. Moreover, the TFM representation of $\Xi(u,0) \triangleq C(sI-A)^{-1}B =P(s)$ and since all real rational strictly proper transfer matrices with no poles on the
imaginary axis form $\mathcal{RL}_2(j\mathbb{R}) \subset \mathcal{L}_2(j\mathbb{R})$ \cite[page 48]{zhou1998essentials}, $ P(s) $ belongs to $ \mathcal{L}_2(j\mathbb{R})\cap \mathbb{R}(s)^{n_y \times n_u}$. The system, $\Xi$, is: P1) a \textit{non-minimum phase} s.t. $Z_R \ne \emptyset$, P2) either $n_y=n_u$ (square) or $n_y \ne n_u$ (nonsquare), and causal ($t > 0 $).
\begin{assumption}\label{ass:minimalABC}
The system given in (\ref{eq:ss0}) is minimal.
\end{assumption}

\begin{problem} \label{pr:problem definition}
{Under Assumption \ref{ass:minimalABC}, let \(\Xi\) denote the forward system given in \eqref{eq:ss0}, and let \(y_o(t)\) be an observed output trajectory. We define the \textit{right inverse} $\Xi_i^\Gamma: \mathbb{Y}\times\mathbb{X}\to\mathbb{U}$, which, given $y_o(t)$ and a prescribed inverse system's initial states, $\bar{x}_0 \in \mathbb{X}$, produces the control input
\begin{align} \label{eq:uinvpr}
u_{inv}(t) = \Xi_{i}^\Gamma(y_o(t),\bar{x}_0), \quad \Vert u_{inv}(t) \Vert_\infty < \infty.
\end{align}
Let $x_0$ be the initial states of the forward system. The error is then defined by
\begin{align} \label{pr:error}
    e(t) &\triangleq  y_o(t)-\Xi(u_{inv}(t),x_0) 
\end{align}
The inverse, $\Xi_i^\Gamma$,} is classified based on $e(t)$ as follows:
\begin{itemize}
   \item[a.]  {Stable exact inverse}: if \( e(t) = 0 \) for all \( t > 0 \).
   \item[b.]  {Stable approximate inverse}: if \( \Vert e(t) \Vert_\infty < \infty \).
\end{itemize}
Here, our goal is to construct a stabilizing inverse operator \( \Xi_i^\Gamma \)such that the error satisfies
\begin{align}\label{ESwithinevitabeerrors}
   \| e(t)\| \leq \alpha \|e(0)\| e^{-\beta t} + \varphi, \quad \forall t > 0,
\end{align}
for some $\alpha, \beta > 0$, and $\varphi \geq 0$.
\end{problem}
\begin{remark} \label{rem:similarproblems}
The terminology here (cf.\ \cite{estrada2007left}) also lets $y_o(t)$ denote a desired trajectory \cite{naderi2018inversion}.  With Assumption~\ref{ass:minimalABC} and $\|e\|_\infty<\infty$, \eqref{pr:error} is Lyapunov-stable. If exact inversion ($\alpha=\varphi=0$) is unattainable, we can employ a robust stabilizing approximate inverse—one viable choice among the many approximations acknowledged in~\cite{romagnoli2019general}—with an irreducible error~$\varphi$.
\end{remark}
\begin{remark}\label{rem:ic}  
As is typical in \textit{exact} stable inversion, we initially assume that \((A, B, C)\) in \eqref{eq:ss0} and the initial states \(x_0\) and  \(\bar{x}_0\) are known. Section~\ref{sec:almostiff} relaxes this assumption by dropping the requirement on \(\bar x_0\).
Section~\ref{sec:practical_modif} goes further by allowing \(x_0\) to be unknown and by also incorporating model uncertainty, resulting in an approximate solution. For further details on initial states, see \cite{Loreto2023}.
\end{remark}

To solve Problem~\ref{pr:problem definition}.a algebraically, define 
\(y_{ic}(t) \triangleq \Xi(0, x_0)\), {$y \triangleq y_o-y_{ic}$, 
\(u_{inv}^{(0)}(t) \triangleq \Xi_i^\Gamma(0, \bar{x}_0)\), 
and \(u(t) \triangleq \Xi_i^\Gamma(y(t), 0)\) 
to represent the trajectories and control inputs for known or zero initial states. 
Then, Remark~\ref{rem:ic} lets us redefine \eqref{eq:uinvpr} as
\(u_{inv}(t) \triangleq \Xi_{i}^\Gamma(y_o(t),\bar{x}_0) - u_{inv}^{(0)}(t)\).} Applying the Laplace transform (see Paley--Wiener Th.\ \cite[p.~104]{hoffman2007banach} for existence) yields

\begin{align}
&\mathscr{L} \{ y_o(t) - {y}_{ic}(t)\} = {C}(s{I}-{A})^{-1}{B}U(s) =Y(s) \label{eq:TFM}\\
& \quad \quad \Rightarrow \ P(s)U(s)=Y(s) \quad s.t. \label{eq:algebraic equation} \\   
&\begin{bmatrix} 
P_{11}(s) & \cdots & P_{1n_u}(s) \\
\vdots & \ddots & \vdots \\
P_{n_y1}(s) & \cdots & P_{n_yn_u}(s)
\end{bmatrix}
\begin{bmatrix}
U_1(s) \\
\vdots \\
U_{n_u}(s)
\end{bmatrix}
=
\begin{bmatrix}
Y_1(s) \\
\vdots \\
Y_{n_y}(s)
\end{bmatrix}. \label{eq:algebraic equation2}
\end{align}
From this point on, similar to \cite[Theorem 8.5.2]{wolovich2012linear}, Problem \ref{pr:problem definition}.a reduces to algebraically solving the rational matrix equation where $\Vert u(t) \Vert_\infty < \infty \iff \Vert u_{inv}(t) \Vert_\infty < \infty$.
\section{Main Results} \label{sec:Main}
\subsection{{Right Inverse}} \label{sec:leftinverse}
Under Assumption \ref{ass:minimalABC}, \eqref{eq:TFM} shows that all rational function entries of $P(s)$ and $Y(s)$ in \eqref{eq:algebraic equation} share the same greatest common denominator, $\Delta=\det (sI-A) \in \mathbb{R}[s]$. To cancel out $1/\Delta$, we multiply \eqref{eq:algebraic equation} by $\Delta$, yielding 
\begin{align} \label{eq:algebraiceq_noGCD}
    \grave{P}(s)U(s)= \grave{Y}(s)
\end{align}
where $U(s) \in \mathbb{R}(s)^{n_u}$ and $\grave{Y}(s) \in \mathbb{R}(s)^{n_y}$. Then, using the \textit{invariant factor theorem} given by \eqref{eq:invariantfactor} on $\grave{P}(s)$ in \eqref{eq:algebraiceq_noGCD}, we have
\begin{align} \label{eq:leftinverse}
U(s)&=[\hat{V}_{R1} \ \hat{V}_{R2}] \begin{bmatrix}
    \Lambda^{-1}(s) & \textbf{0}_{r \times (n_y-r)}\\
    \textbf{0}_{(n_u-r)\times r} & \textbf{0}_{(n_u-r)\times (n_y-r)}
\end{bmatrix}\begin{bmatrix}
    \hat{U}_{R1} \\
    \hat{U}_{R2}
\end{bmatrix}\grave{Y}(s)\nonumber\\& \qquad +(I_{n_u}-\hat{V}_{R1}{V}_{R1})\kappa \nonumber \\
&=\Biggr( [\hat{V}_{R1} \ \hat{V}_{R2}]  \begin{bmatrix}
    \Lambda^{-1}(s) & \textbf{0}_{r \times (n_y-r)}\\
    \textbf{0}_{(n_u-r)\times r} & \textbf{0}_{(n_u-r)\times (n_y-r)}
\end{bmatrix}\begin{bmatrix}
    \hat{U}_{R1} \\
    \hat{U}_{R2}
\end{bmatrix} \nonumber\\ & \qquad +(I_{n_u}-\hat{V}_{R1}{V}_{R1})\bar{\kappa}  \Biggr)\grave{Y}(s)= \Xi_i^{\Gamma}(s) \grave{Y}(s) 
\end{align}
where $\Lambda^{-1}(s)=\text{diag}[{1}/{d_1(s)} \ \dots {1}/{d_r(s)}] \in \mathbb{R}(s)^{r \times r}$,  $\hat{U}_R(s) \in \mathbb{R}[s]^{n_y \times n_y}$ and $\hat{V}_R(s) \in \mathbb{R}[s]^{n_u \times n_u}$ are $\mathcal{R}$-Unimodular matrices such that  ${U}_R(s)\hat{U}_R(s)=I_{n_y}$ and ${V}_R(s)\hat{V}_R(s)=I_{n_u}$ respectively, and $\kappa, \bar{\kappa} \in \mathcal{R}\mathcal{H}_\infty \  (\bar{\kappa}\triangleq \kappa(\grave{Y}^T(s)\grave{Y}(s))^{-1}\grave{Y}^T(s))$ are any arbitrary vectors. 
{\begin{remark}
From an algebraic standpoint, although obtaining $U(s)$ involves multiplying $\grave{P}(s)$ by its left inverse, following \cite{moylan1977stable,estrada2007left}, we denote $\Xi_i^\Gamma$ as a “\textit{right inverse}”.
\end{remark}}
\subsection{Algebraic Necessary and Sufficient Conditions} \label{sec:algebraic_iff}
In this subsection, to find the solution(s) for Problem \ref{pr:problem definition}.a, we will show the solvability of (\ref{eq:algebraic equation}) algebraically and therefore the existence and boundedness conditions of \eqref{eq:leftinverse}. 

\begin{theorem} \label{th:algebraic_iff}
Consider $\Xi$ in \eqref{eq:ss0} with P1-P2 and $\Xi^\Gamma_i$ in \eqref{eq:leftinverse}. Then necessary and sufficient conditions for Problem \ref{pr:problem definition}.a. are
\begin{align*}
    \Bigl( e(t)=0, \ \Vert u(t) \Vert_\infty < \infty \Bigl) \iff & \Bigl( Y(s) \in \im_{\mathbb{R}(s)}(P) \Bigl) \text{ and } \\
     & \Bigl( Y(z) = 0, \ \forall z \in Z_R \Bigl).
\end{align*}
\end{theorem}
\begin{proof}
Note that $e(t)=0$ implies \eqref{eq:algebraic equation} and \eqref{eq:algebraiceq_noGCD} hold and vice versa. We can now proceed to the proof for both directions.

$\implies$Suppose $\exists U(s) \in \mathcal{R}\mathcal{H}_\infty$ satisfying $e(t)=0$.

(1:) $e(t)=0$ means \eqref{eq:leftinverse} holds. Then, pre-multiplying \eqref{eq:leftinverse} by $ \grave{P}(s) = U_R(s)S_m(s)V_R(s)$ and simplifying yields
\begin{align} \label{eq:Y invaiance}
    U_{R1}\hat{U}_{R1}\grave{Y}(s)=\grave{Y}(s)
\end{align}
which means $\grave{Y}(s)$ is invariant under the orthogonal projection onto $\im_{\mathbb{R}(s)}(P)$. Thus, we get $Y(s) \in \im_{\mathbb{R}(s)}(P)$.

(2:) Note that $(I_{n_u}-\hat{V}_{R1}{V}_{R1})$ in \eqref{eq:leftinverse} is pure polynomial and $\kappa \in \mathcal{R}\mathcal{H}_\infty$. Therefore, $(I_{n_u}-\hat{V}_{R1}{V}_{R1})\kappa$ does not contribute any unstable solution(s) to $U(s)$. So, considering only the $\hat{V}_{R1}(s)\Lambda^{-1}(s)\hat{U}_{R1}(s)\grave{Y}(s)$ is enough for stability. Let's consider two cases on $\grave{P}(s)$; C1: $\rank_{\mathbb{R}(s)}(\grave{P})=\min(n_u,n_y)$, C2: $\rank_{\mathbb{R}(s)}(\grave{P})<\min(n_u,n_y)$. 
For C2, re-writing \eqref{eq:leftinverse} and ignoring $(I_{n_u}-\hat{V}_{R1}{V}_{R1})\kappa$, we get
\begin{align}
    &U(s)=\hat{V}_{R1}\Lambda^{-1}\hat{U}_{R1}\grave{Y}(s) \implies
    \Lambda\hat{V}^\ell_{R1}U(s)=\hat{U}_{R1}\grave{Y}(s)\nonumber \\
    &\implies U_{R1}\Lambda\hat{V}^\ell_{R1} U(s)=U_{R1}\hat{U}_{R1}\grave{Y}(s) = \grave{Y}(s) \label{eq:equivalencyofnoGCD}
\end{align}
where there always exist $\hat{V}^\ell_{R1}$ satisfying $\hat{V}^\ell_{R1}\hat{V}_{R1}=I_r$. Now assume that either $\grave{P}(s)$ or $U_{R1}\Lambda\hat{V}_{R1}^\ell$ is not a minimal basis. The minimal basis \( \grave{P}_m(s) \) of an \( n_y \)-dimensional vector space of \( n_u \)-tuples over \( \mathbb{R}[s] \) for \eqref{eq:algebraiceq_noGCD} is given by \( \grave{P}_m(s) = L_T(s)\grave{P}(s) \) for condition C1, and \( \grave{P}_m(s) = L_T(s)U_{R1}\Lambda\hat{V}_{R1}^\ell \) for condition C2. Here, the \( \mathcal{R}\)-Unimodular transformation \( L_T(s) \in \mathbb{R}[s]^{n_y \times n_y} \), which yields the minimal basis, exists with constant (degree zero) determinant (see \cite[Remark 2]{forney1975minimal}). Then, left multiplying $L_T(s)$ to \eqref{eq:algebraiceq_noGCD} or \eqref{eq:equivalencyofnoGCD} results in
\begin{align} \label{eq:LTtransform_expanded}
 \grave{P}_m(s)U(s)=L_T(s)\grave{Y}(s)
\end{align}
\noindent Now define $\grave{Y}(s) \triangleq \frac{1}{\Delta_Y}{\grave{Y}_N}(s)$ such that $\grave{Y}_N(s) \in \mathbb{R}[s]^{n_y}$, and the scalar $\Delta_Y \in \mathbb{R}[s]$. So, multiplying \eqref{eq:LTtransform_expanded} by $\Delta_Y$ yields
\begin{align} \label{eq:LTtransform_expanded_purepoly}
\grave{P}_m(s)\Delta_YU(s)=L_T(s)\grave{Y}_N(s)
\end{align}
Note that $\grave{P}_m(s)$ is a minimal polynomial basis and the right-hand side of \eqref{eq:LTtransform_expanded_purepoly} is pure polynomial, so based on \textit{Adapted Forney's Theorem}, $\Delta_YU(s)$ must be polynomial. It means that $\Delta_Y$ captures all poles of $U(s)$. Moreover, non-singular transformations of $\grave{P}(s)$ preserves the invariant zeros \cite{misra1994computation} so that $y^*_z\grave{P}_m(z)=0u^*_z=0$. Left multiplying $y^*_z$ to \eqref{eq:LTtransform_expanded_purepoly} yields
\begin{align} 
\label{eq:LTtransform_expanded_purepoly_diraction}
\underbrace{y^*_z\grave{P}_m(z)}_{=0}\Delta_Y(z)U(z)=\underbrace{y^*_zL_T(z)}_{\ne 0}\grave{Y}_N(z)
\end{align}
Since the left-hand side of \eqref{eq:LTtransform_expanded_purepoly_diraction} becomes zero, $\det(L_T) = \text{constant} \ne 0$ (by polynomial-unimodularity), and $1/\Delta_Y$ is stable because $U(s) \in \mathcal{RH}_\infty$ is assumed in this direction, the only way to satisfy the equality is $\grave{Y}_N(z)=0$ implying $Y(z)=0$. 

$\impliedby$ Suppose that we have $Y(s) \in \im_{\mathbb{R}(s)}(P)$ and $Y(z)=0, \forall z \in Z_R$. For C1 \& C2, since $Y(s) \in \im_{\mathbb{R}(s)}(P)$, we have $\rank_{\mathbb{R}(s)}(P)=\rank_{\mathbb{R}(s)}([P\ : \ Y])$, which shows that we always have solution(s) for $U(s)$ in \eqref{eq:algebraiceq_noGCD} yielding $e(t)=0$ in \eqref{pr:error}, whether $U(s)$ is stable or not.

For boundedness, consider \eqref{eq:rhpzeros} and rewrite \eqref{eq:algebraiceq_noGCD} as 
\begin{align} \label{eq:algebraiceq_noGCD_Ufactored}
    \grave{P}(s)U(s) = \grave{P}(s)\frac{U_N(s)}{U_D(s)} = \grave{Y}(s) \\
    \grave{P}(s)U_N(s)=U_D(s)\grave{Y}(s) \label{eq:algebraiceq_noGCD_Ufactored2}
\end{align}
where $U_N(s) \in \mathbb{R}[s]^{n_u}$ and scalar $U_D(s) \in \mathbb{R}[s]$. Now define a new scalar $z^+(s)$ which contains all RHP zeros of $\grave{P}(s)$ as
 \begin{align*} 
    z^+(s)=\prod_{z \in Z_R \bigcap \mathbb{R}}(s - z)^{m_z}\prod_{z \in Z_R\bigcap (\mathbb{C}\setminus \mathbb{R})}(s - z)^{m_z^\ast} (s - {z}^\ast)^{m_z^\ast}
\end{align*} 
where $z^+(s) \in \mathbb{R}[s]$, $m_z$ and $m_z^\ast$ denote the multiplicities of $z$ and ${z}^*$ respectively. Thus, there exist non-zero vector $y^\ast_z \ (\text{or } u_z)$ s.t. $y^\ast_z\grave{P}(z)=0 $ (or $ \grave{P}(z)u_z=0.y_z$) $\forall z \in Z_R$. 

Now suppose $\det(U_D(\bar{z})) = 0$ for some $\bar{z}$ where $\Re(\bar{z})>0$ which means $U(s)$ is unstable. Either, suppose $(s-\bar{z}) \in z^+(s)$, then left multiplying $y^\ast_{\bar{z}}$ to \eqref{eq:algebraiceq_noGCD} and evaluating at $s=\bar{z}$ yield
\begin{align}\label{eq:contra1_th1_nec}
  \Biggl( \underbrace{(s-\bar{z})\begin{bmatrix}
\ast(s)_1 \\
\vdots \\
\ast(s)_{n_y}
\end{bmatrix}}_{= y^\ast_{\bar{z}}\grave{P}(s)}\frac{U_N(s)}{\underbrace{(s-\bar{z})U'_D(s)}_{=U_D(s)}} \Biggr) (\bar{z}) \ne 0 = y^\ast_{\bar{z}}\underbrace{\grave{Y}(\bar{z})}_{= 0}
\end{align}
or suppose $(s-\bar{z}) \notin z^+(s)$ then left multiplying $y^\ast_{\bar{z}}$ to \eqref{eq:algebraiceq_noGCD_Ufactored2} and evaluating at $s=\bar{z}$ yield
\begin{align} \label{eq:contra2_th1_nec}
  \underbrace{y^*_{\bar{z}} \grave{P}(\bar{z})}_{\ne 0}{U_N(\bar{z})} \ne 0 = {y^*_{\bar{z}}} \grave{Y}(\bar{z}) \underbrace{U_D(\bar{z})}_{= 0}
\end{align}
\eqref{eq:contra1_th1_nec} and \eqref{eq:contra2_th1_nec}
lead to a contradiction. Thus
$U(s)$ includes only LHP poles. 
The proof is now complete. 
\end{proof}
\begin{remark}
While the algebraic equation is inspired by Exact Model Matching, we explicitly define necessary and sufficient conditions on the system's output, applicable even to non-minimum phase and singular systems. {Specifically, the condition \(Y(s) \in \im_{\mathbb{R}(s)}(P)\) defines the set of reachable outputs, while \(Y(z) = 0\) for all \(z \in Z_R\) defines the set of outputs for which a corresponding stable input exists. Relaxing either condition results in losing the property `$=0$'' over the $e(t)$.}
\end{remark}
\begin{remark}
For the counterpart involving left inverses in contexts such as fault detection, see \cite{Loreto2023,estrada2007left}. These works give necessary and sufficient conditions for invertibility under the assumptions of full-rank $P(s)$, $D$ and with $Z_R=\emptyset$.   
\end{remark}
Assuming $Z_R = \emptyset$ yields the following corollaries:
\begin{corollary} \label{coro:min_phase_inverses}
Assume $P(s)$ is square and full rank so that $U_{R}=U_{R1},V_{R}=V_{R1}, (I_{n_u}-\hat{V}_{R1}{V}_{R1})=0$, and $ Z_R = \emptyset$, then \eqref{eq:leftinverse} becomes $U(s) ={\hat{V}_{R1}\Lambda^{-1}(s)\hat{U}_{R1}}Y(s)$.
\end{corollary}
\begin{corollary} \label{coro:min_phase_left_pseudo_inverses}
Assume $P(s)$ is non-square and full rank, so that $V_{R}=V_{R1}, (I_{n_u}-\hat{V}_{R1}{V}_{R1})=0$, and $ Z_R = \emptyset$, then \eqref{eq:leftinverse} becomes $U(s) = \begin{bmatrix}
    \hat{V}_{R1}\Lambda^{-1}(s) &
   \textbf{0}_{n_u \times (n_y-n_u)}
\end{bmatrix}
\begin{bmatrix}
    \hat{U}_{R1} \\
    \hat{U}_{R2}
\end{bmatrix}Y(s)
= (P^*(s)P(s))^{-1}P^*(s)Y(s) = P^\dagger(s)Y(s)$.

\end{corollary}
\subsection{Almost Necessary and Sufficient Conditions}\label{sec:almostiff}
In this section, we propose approximate remedies for non-minimum phase systems. Specifically, by using \(\mathcal{H}_\infty\)-theory, we relax the condition 
\(\bigl(Y(z) = 0,\ \forall z \in Z_R\bigr)\) 
yielding 
\begin{align*}
   ( e(t)\to 0, \ \Vert u(t) \Vert_\infty < \infty ) \iff  ( Y(s) \in \im_{\mathbb{R}(s)}(P) )
    \end{align*}
for any unknown initial state, \(\bar{x}_0\), of the inverse system.

Assumption~\ref{ass:minimalABC} ensures that we can assume $P(s)$ is stable, either inherently or via stabilization, thus avoiding RHP pole/zero cancellations between the physical components on the feedforward path \cite[Section 4.7.1]{skogestad2005multivariable}, as is typical in stable right inversion \cite{romagnoli2019general, devasia2002should}. Note also that a stable (or stabilized) forward system $\Xi$, thus $P(s)$, in \eqref{eq:ss0} produces an unbounded $y(t)$ \textit{if and only if} $u(t)$ is unbounded. Therefore, we assume \(Y(s) \in \mathcal{RH}_\infty\) to ensure well-behaved operation.

Now, we define a \textit{virtual} loop in the sense of classical feedback structure shown by Fig.\ref{f:blockdiag}.(a). Given the system $P(s)$ as in \eqref{eq:algebraic equation}, suppose we have a virtual controller, say $K(s) \in \mathbb{R}(s)^{n_u \times n_y}$. The key transfer functions within the virtual feedback loop are defined as follows: \( T_i(s) \triangleq (I_{n_u} + K(s)P(s))^{-1}K(s)P(s) \); \( S_i(s) \triangleq (I_{n_u} + K(s)P(s))^{-1} \); and \( L_i(s) \triangleq K(s)P(s) \). By breaking the loop at the \textit{output}, we have $T_o,S_o,L_o$.
\begin{theorem}[Internal Stability, \cite{zhou1998essentials}] \label{th:internal_stab}
Under Assumption \ref{ass:minimalABC}, the closed-loop virtual system is internally stable iff
\begin{align}\label{eq:internal_stab_matrix}
    \begin{bmatrix}
        I & K\\
        -P & I
    \end{bmatrix}^{-1} \in \mathcal{R}\mathcal{H}_\infty
\end{align}
\end{theorem}
Now consider augmenting the performance weight ${W}_{P}$, ${W}_{P}~\triangleq \text{diag} [ {W_{P1}}(s) \ {W_{P2}}(s) \ \dots \ {W_{Pn_y}}(s)]$ to form the linear fractional transformation (LFT) as
\begin{align} \label{eq:augmentedG}
\begin{bmatrix}
        z(t)\\
    e_\upsilon(t)
\end{bmatrix}
={\begin{bmatrix}
    W_P & W_PP\\
    -I & -P
\end{bmatrix}}\begin{bmatrix}
    w(t)\\
  u_\upsilon(t)
\end{bmatrix}=G\begin{bmatrix}
    w(t)\\
  u_\upsilon(t)
\end{bmatrix}
\end{align}
where each scalar in $W_P$ is as defined in \cite[Equation 22]{kurkccu2018robust}, $e_\upsilon(t) \triangleq r_\upsilon(t)-y_\upsilon(t)$, and $z(t)\triangleq W_Pe(t)$. The LFT is then given by $\mathcal{F}_l(G,K)\triangleq {G}_{11}+{G}_{12}{K}({I}-{G}_{22}{K})^{-1}{G}_{21}$ where the $\mathcal{H}_\infty$ control problem involves finding $K$ such that
\begin{align}\label{eq:min_maxSV}
    \minimize_K\ &\max_{\omega}\bar{\sigma}\Bigl({G}_{11}+{G}_{12}{K}({I}-{G}_{22}{K})^{-1}{G}_{21}\Bigl)(j\omega) \nonumber\\
    \subj2 & \begin{bmatrix}
        I & K\\
        -P & I
    \end{bmatrix}^{-1} \in \mathcal{R}\mathcal{H}_\infty.
\end{align}
Respecting the analytic limits (waterbed/Bode, logarithmic-integral and interpolation bounds) in \cite{chen2000logarithmic} and leveraging the Youla-based convexification of \eqref{eq:min_maxSV} described in \cite[Sec.~3.3]{anderson201639system}, we ensure \(
\| \mathcal{F}_l(G,K) \|_\infty = \gamma < \infty .
\)

By letting \(\epsilon_i \to 0\), we can replace all approximate integrators in \(K\) with pure integrators. Moreover
If \(\bar{\sigma}(S_i(j(0,\bar{\omega}])) < \mathbb{E}[N^2]\), then \(\bar{\sigma}(S_i(j(0,\bar{\omega}])) = 0\) where $\mathbb{E}[N^2]$ denotes the expected value mean square of random noise in dB .
\begin{remark}\label{rem:no knowledge synthesis}
The approximate inverse described by the following theorems can be constructed with no knowledge of $\bar{x}_0$.
\end{remark}
\begin{theorem}\label{th:nm_phase_inverses}
Consider \( P(s) \) as defined in \eqref{eq:algebraic equation}-\eqref{eq:algebraic equation2} with \(\text{rank}_{\mathbb{R}(s)}(P) = r = n_y \leq n_u \), and \( Z_R \neq \emptyset \). Define the approximate inverse \(\Xi_a^\Gamma:\mathbb{Y}\times \mathbb{X} \to \mathbb{U}\) with unknown $\bar{x}_0$ as
\begin{equation}
\label{eq:almostinverse}
\Xi_a^\Gamma(s) \;\triangleq\; S_i(s)K(s)
\;=\;\small \left(\begin{array}{c|c}
   A_i^a & B_i^a \\ \hline
   C_i^a & D_i^a
\end{array}\right)\normalsize,
\end{equation}
then, by noting \eqref{eq:TFM} and rewriting \eqref{eq:uinvpr}, the control input is
\begin{equation} \label{eq:almostinverse_time}
  u_{inv}(t) = \mathscr{L}^{-1}\{S_i(s)K(s)Y(s)\} + \Xi_a^\Gamma(0,\bar{x}_0) \,
\end{equation}
which satisfies $\|u_{inv}(t)\|_\infty<\infty$ and
\[
\|e(t)\| \leq \alpha \|e(0)\| e^{-\beta t}, \ \forall t > 0, \text{ for some } \alpha,\beta>0. \]
\end{theorem}
\begin{proof}
We decompose the control input as
\begin{align}
u_{inv}(t) &= u_{inv}^{(f)}(t) + u_{inv}^{(0)}(t) = \Xi_a^\Gamma(y(t),0) + \Xi_a^\Gamma(0,\bar{x}_0) \nonumber \\&=\mathscr{L}^{-1}\{S_i(s)K(s)Y(s)\}+\Xi_a^\Gamma(0,\bar{x}_0),
\end{align}
From \eqref{eq:TFM}, taking the Laplace transform of \(e(t)\) yields
\begin{align}\label{EinS}
E(s) \;=\; Y(s) - P(s)\,U_{inv}(s).
\end{align}
Substituting
\(
U_{inv}(s) = S_i(s)K(s)Y(s) + U_{inv}^{(0)}(s)
\)
and noting \(S_i(s) = \bigl(I + K(s)\,P(s)\bigr)^{-1}\), we obtain
\begin{align*}
E(s) = Y(s) &- P(s)\,\bigl(I + K(s)\,P(s)\bigr)^{-1}K(s)\,Y(s)
\\&-P(s)\,U_{inv}^{(0)}(s).
\end{align*}
Define 
\(
T_o(s) \triangleq \bigl(I + PK\bigr)^{-1}PK
\)
and 
\(
S_o(s) \triangleq I - T_o(s),
\)
so 
\[
E(s) \;=\; Y(s)\,S_o(s) \;-\; P(s)\,U_{inv}^{(0)}(s).
\]
By Theorem~\ref{th:internal_stab}, we have \(S_i(s)K(s) \in \mathcal{RH}_\infty\) and \(S_o(s)\in\mathcal{RH}_\infty\), implying the realization of \(\Xi_a^\Gamma\) is stable. Hence 
\begin{align} \label{uinv(0)}
\lim_{t\to\infty} C_i^a\,e^{A_i^a\,t}\,\bar{x}_0 = 0 \implies \lim_{s\to0}sP(s)\,U_{inv}^{(0)}(s) =0. 
\end{align}
Consequently, since $Y(s)\in \mathcal{RH}_\infty$, \(\mathscr{L}^{-1}\{Y(s)\,S_o(s)\}\) decays exponentially, and \eqref{uinv(0)}, we have 
\begin{align*}
\|e(t)\| \le \alpha\,\|e(0)\|\,e^{-\beta t},\ \forall t>0,
\end{align*}
for \(\alpha,\beta>0\). Finally, \(\mathscr{L}^{-1}\{S_i(s)K(s)Y(s)\}\) is bounded (since \(S_i(s)K(s), Y(s)\in\mathcal{RH}_\infty\)), and since \(\Xi_a^\Gamma(0,\bar{x}_0)\) is finite, so \(\|u_{inv}(t)\|_\infty<\infty\). This completes the proof.
\end{proof}

Now let us define $P_p(s) \in \mathbb{R}(s)$ which is a scalar approximation of MIMO $P(s)$ in \eqref{eq:algebraic equation}. To do it first consider one entry of $P(s)$ for $n_u<n_y$ (taking an average of an entire column is also an option). Because of Assumption \ref{ass:minimalABC}, the characteristic equation of $P_p(s)$, $\Delta_p$, equals $\Delta$. Then, if for any RHP zeros of $P(s)$ does not have an RHP zero of $P_p(s)$, zero augmentation as $P_p(s)=z^+(s)P_p(s)$ is employed. Then, substitute $P$ with $P_p$ in (\ref{eq:augmentedG}) and by letting $\epsilon_i \to 0$ solve (\ref{eq:min_maxSV}) to get stabilizing $K_p(s) \in \mathbb{R}(s)$, virtual loop's control system  over the scalar approximation $P_p(s)$. Then the scalar complementary sensitivity function for this modified virtual loop is given by 
\begin{align}\label{eqTp}
T_p(s) \triangleq \Bigr( P_p(s)K_p(s)\Bigl)/\Bigr( 1+P_p(s)K_p(s)\Bigl)
\end{align}
\begin{theorem}\label{th:nm_phase_pseudoinverses}
Consider \(\text{rank}_{\mathbb{R}(s)}(P) = n_u < n_y\) and \( Z_R \neq \emptyset \). Then, for unknown $\bar{x}_0$, the approximate inverse is 
\begin{equation} \label{eq:almostpinverse}
  \Xi_a^\Gamma(s) = T_p(s)P^{\dagger}(s)=\;
\small \left(\begin{array}{c|c}
   A_i^a & B_i^a \\ \hline
   C_i^a & D_i^a
\end{array}\right)\normalsize,
\end{equation}
and the corresponding control input is
\begin{equation} \label{eq:almostinverse_ns_time}
  u_{inv}(t) = \mathscr{L}^{-1}\{T_p(s)P^{\dagger}(s)Y(s) \}+ \Xi_a^\Gamma(0,\bar{x}_0).
\end{equation}
Assuming $Y(s) \in \im_{\mathbb{R}(s)}(P)$, then $u_{inv}(t)$ satisfies
\begin{equation} \label{exp_decay_nmppseudo}
    \|e(t)\| \leq \alpha \|e(0)\| e^{-\beta t}, \forall t > 0,
\end{equation}
for some $\alpha,\beta>0$, while ensuring $\|u_{inv}(t)\|_\infty<\infty$. 
\end{theorem} 
\begin{proof}
Consider the following decomposition as
\begin{align}
u_{inv}(t) = u_{inv}^{(f)}(t) + u_{inv}^{(0)}(t),
\end{align}
and note that $Y(s) \in \im_{\mathbb{R}(s)}(P) \implies U_{R1}\hat{U}_{R1}{Y}(s)={Y}(s) \implies P(s)P^\dagger(s)Y(s)=Y(s)$ and $P^{\dagger}(s)$, as in Corollary \ref{coro:min_phase_left_pseudo_inverses} solves \eqref{eq:algebraic equation} but with an unstable $u_{inv}(t)$ since \( Z_R \neq \emptyset \). To get a stable $u_{inv}(t)$, we can substitute $U_{inv}(s)=T_p(s)P^{\dagger}(s)Y(s)+U_{inv}^{(0)}(s)$ into \eqref{EinS} as
\begin{equation*}
 E(s)=Y(s)-P(s)T_p(s)P^\dagger(s)Y(s) -P(s)U_{inv}^{(0)}(s)   
\end{equation*}
and since $T_p(s)$ is scalar, using $P(s)P^\dagger(s)Y(s)=Y(s)$, and $S_p(s) \triangleq I-T_p(s)$, we get 
\begin{align}
 E(s)=Y(s)S_p(s)-P(s)U_{inv}^{(0)}(s).  
\end{align}
Since $K_p$ in \eqref{eqTp} is obtained by \eqref{eq:min_maxSV}, which gives \footnotesize $\begin{bmatrix}
        I & K_p\\
        -P_p & I
\end{bmatrix}^{-1} \in \mathcal{R}\mathcal{H}_\infty$\normalsize, $P(s)$ and $P_p(s)$ shares the same poles, we have $S_p(s) \in \mathcal{RH}_\infty$.
In addition, \eqref{eq:internal_stab_matrix} yields that If $P(s)$ has a RHP zero at $z$, then $P^\dagger(s)$ has a RHP-pole at $z$ and $PK(I+PK)^{-1}$, has a RHP-zero at $z$ \cite{skogestad2005multivariable}, which applies on $T_p(s)$. Thus, $T_p(z)=0$ yielding $T_p(s)P^{\dagger}(s)Y(s) \in \mathcal{RH}_\infty$
Then following the proof of Theorem \ref{th:nm_phase_inverses} yields
\(
\|e(t)\| \leq \alpha \|e(0)\| e^{-\beta t}, \quad \forall t>0.
\)
and $\Vert u_{inv}(t) \Vert_\infty < \infty$. This completes proof.
\end{proof} 
\begin{remark}
Although operators such as $S_i(s)K(s)$, $S_i(s)$, and $T_i(s)$ are commonly used to analyze standard closed-loop responses, here we \emph{repurpose} $S_i(s)K(s)$ (\text{in Theorem}~\ref{th:nm_phase_inverses}) or $T_p(s)P^\dagger(s)$ (\text{in Theorem}~\ref{th:nm_phase_pseudoinverses}) as a direct control action - rather than using $K(s)$ solely like a conventional loop, which departs from typical closed-loop treatments. Consequently, while alternative methods can also be used to design the $K(s)$ - and therefore $S_i(s)K(s)$ and $T_p(s)$ - the $\mathcal{H}_\infty$ framework provides valuable analytical properties. In particular, it yields an “optimal” approximate inverse within our setting.
\end{remark}
The next corollary provides another approximate inverse, derived from Theorem~\ref{th:algebraic_iff}, that accommodates singular systems. To achieve this, one may, for instance, introduce sufficient low-pass characteristics, or alternatively apply a scalar output redefinition $z^+/z_d^+$, where 
\begin{align*} 
    z^+_d(s)=\prod_{z \in Z_R\bigcap \mathbb{R}} (s + z)^{m_z} \prod_{z \in Z_R\bigcap (\mathbb{C}\setminus \mathbb{R})} (s + z)^{m_z^\ast} (s + {z}^\ast)^{m_z^\ast}
\end{align*}
\begin{corollary} \label{coro:modified Th1}
Consider \(\text{rank}_{\mathbb{R}(s)}(P) \le \min (n_u,n_y) \) and \( Z_R \neq \emptyset \) with $Y(s) \in \im_{\mathbb{R}(s)}(P)$. Then, the approximate inverse is 
\begin{equation} \label{eq:modified Th1 approx}
  \Xi_a^\Gamma(s) = \Xi_i^\Gamma(s)({z^+(s)}/{z_d^+(s)}).
\end{equation}
satisfying Problem~\ref{pr:problem definition}.b with $\varphi=0$
\end{corollary}
\begin{remark}
Theorems~\ref{th:nm_phase_inverses}, \ref{th:nm_phase_pseudoinverses}, and Corollary~\ref{coro:modified Th1} each provide a stable approximation when $Y(z) \neq 0$ for all $z \in Z_R$. Note also that the control inputs $u_{\text{inv}}(t)$ defined by \eqref{eq:almostinverse_time}, \eqref{eq:almostinverse_ns_time}, and \eqref{eq:modified Th1 approx} are not unique.
In Theorem~\ref{th:nm_phase_inverses}, the condition $Y(s) \in \im_{\mathbb{R}(s)}(P)$ always holds because $P(s)$ is either square or overactuated and has full column rank. However, Theorem~\ref{th:nm_phase_pseudoinverses} and Corollary~\ref{coro:modified Th1} allow for $Y(s)\notin \im_{\mathbb{R}(s)}(P)$, resulting in $\varphi \neq 0$ in \eqref{ESwithinevitabeerrors}.
\end{remark}
All results thus far assume a nominal system \( P(s) \), but modeling errors, numerical problems, and unknown forward system initial states can lead to instabilities. These implications, along with structures over \( \varphi \), will be discussed in the next section. 
\subsection{Robustness} \label{sec:practical_modif}
Stable inversion is typically studied under the assumption of complete model knowledge including initial states $x_0$ which is often not feasible. In this section, we further 
relax the constraints in Theorem~\ref{th:algebraic_iff}. Specifically, Theorem~\ref{th:feedback_unc} addresses the condition $Y(s)\notin \im_{\mathbb{R}(s)}(P)$ under uncertainties, while Corollaries \ref{coro:squareextunc}-\ref{coro:relaxationsUNC} explore 
the full range of possible relaxation scenarios. Now, the desired output can be rewritten as
\begin{align} \label{eq:algebraicequation_desired}
P(s)U(s)=Y_d(s) 
\end{align}
where bounded $Y_d(s)\in \im_{\mathbb{R}(s)}(P) \cap \mathcal{RH}_\infty$ denotes the \textit{desired} output. For the uncertainties over $P(s)$ in \eqref{eq:algebraicequation_desired}, let the perturbed plant, $P_\Pi(s)$, be a member of all possible plants
\begin{equation} \label{eq:phat}
P_\Pi(s) = \small \left(\begin{array}{c|c}
   A_\Pi & B_\Pi \\ \hline
   C_\Pi & D_\Pi
\end{array}\right) \normalsize \in  \Pi \triangleq \left\{ ({I}+W_1 \Delta_u {W}_2){P} \right\}
\end{equation}
where $W_1, W_2$ are TFMs that characterize the spatial and frequency structure of the uncertainty, $\Delta_u$ denotes any unknown unstructured function with 
$\Vert \Delta_u \Vert_\infty <1$ \cite[chapter 8.1]{zhou1998essentials}. Moreover, the perturbed (real) system's zeros are
\begin{align} \label{eq:zerosPPi}
\text{Z}_{P_\Pi} &\triangleq \{z \in \mathbb{C} : P_\Pi(z)u_z=0y_z\}\\
\text{Z}_{R_\Pi} &\triangleq  \{z \in \text{Z}_{P_\Pi} : \Re(z) > 0\} \subseteq \text{Z}_{P_\Pi}.  \label{eq:rhpzerosPpi}
\end{align}
Here ${Z}_{P_\Pi}$ might be different from $Z_P$ which means $P_\Pi$ in (\ref{eq:phat}) considers \textit{uncertain RHP zeros} for $P(s)$. Then, the perturbed output for a given input as
\begin{align}
    y_\Pi(t)=\Laplace^{-1}\{ P_\Pi(s)U(s) \} =\Laplace^{-1}\{ Y_\Pi(s) \}.
\end{align}

Note that the inner product for any vector-valued functions $F,G \in \mathcal{L}_2(j\mathbb{R}) \cap \mathbb{R}(s)^{n_y}$ is defined as:
\begin{align}
    \langle F,G \rangle \triangleq \frac{1}{2\pi}\int_{-\infty}^\infty \sum_{k=1}^{n_y} F_k^\ast(j\omega)G_k(j\omega)d\omega
\end{align}
Consider \(L\) given in \eqref{eq:Lind_columns_P}, as $ P(s) \in \mathcal{L}_2(j\mathbb{R})\cap \mathbb{R}(s)^{n_y \times n_u}$, $L \in \mathcal{L}_2(j\mathbb{R})$. Given that the columns in \(L\) are not necessarily orthogonal, we can rectify this by applying the Gram-Schmidt process to the columns of \(L\). This process can be expressed as:
\begin{align}
v_i = p_i - \sum_{j=1}^{i-1} \frac{\langle p_i, q_j \rangle}{ \langle q_j,q_j \rangle } q_j,  q_i = \frac{v_i}{\sqrt{\langle v_i,v_i \rangle}},   i = 1, \ldots, r.
\end{align}
Here, the \(v_i\) vectors are orthogonal, and the \(q_i\) vectors are orthonormal. 
Based on this, the orthonormal set is defined as ${Q} = [ q_{1} \ \dots \ q_{r} ]\in \mathcal{L}_2(j\mathbb{R}) \cap \mathbb{R}(s)^{n_y \times r}$ which spans the same subspace as \(L\) in \eqref{eq:Lind_columns_P}. With these conditions the transformation from \(L\) to \(Q\) always exists. 
\begin{definition} \label{def:projres}
Let \(\im_{\mathbb{R}(s)}(P)\) be spanned by an orthonormal basis $\{q_k\}_{k=1}^{r} \subset \mathcal{L}_2(j\mathbb{R})$. For any $F \in \mathcal{L}_2(j\mathbb{R})$, the projection onto $\im_{\mathbb{R}(s)}(P)$ is given by:
\begin{align} 
\label{eq:Profoveractout}
    \proj_{\im{(P)}}[ &F ] = \sum_{i=1}^r \langle F,q_i\rangle q_i,  
\end{align}
where \( F \) can be decomposed as:
\begin{align} \label{eq:projdecomp}
&F = \proj_{\im(P)}[F] + \res[F],\\
\text{with }&\langle \res[F], q_k \rangle = 0,  \forall k \in \{1, \dots, r\}. \label{eq:projdecomp2}
\end{align}
\end{definition}
Then, the overall feedback strategy combining two feedforward actions to deal with uncertainties is given by Fig. \ref{f:blockdiag}. Here, $u_{ff}(t)$ is feedforward control input, $u_{fb}(t)$ is feedback control input, and $u_c(t)$ is combined (effective) control input $u_c(t)$, $y_\Pi(t)$ denotes the real output under $u_c(t)$, and we have $y_\Delta(t) \triangleq y_\Pi(t)-y(t)$.

Based on the conditions, Corollary~\ref{coro:modified Th1} (Fig. \ref{f:blockdiag}(c)(i)) is valid for all system classes, including non-minimum phase, square, non-square, and singular systems, though it requires an output redefinition in the feedback path, $\bar{y}_\Delta(t)~=~({z^+}/{z_d^+})~\ast~y_\Delta(t)$.

The other approximate inversions are applicable to full-rank systems; for square/overactuated systems, the loop corresponds to Fig. \ref{f:blockdiag}(c)(ii), and for non-square (underactuated) systems, it corresponds to Fig. \ref{f:blockdiag}(c)(iii). Then, over the nominal system $P(s)$, the feedforward control signal $u_{ff}(t)$ is obtained by solving \eqref{eq:leftinverse}, \eqref{eq:almostinverse}, or \eqref{eq:almostpinverse}, subject to \eqref{eq:algebraicequation_desired}. Similarly, $u_{fb}(t)$ and $u_{ff}(t)$ are obtained by following the same procedures.
The next question, whether the proposed loop in Fig.\ref{f:blockdiag} can compensate the error caused by uncertainty, is revealed by the upcoming theorem.
\begin{figure}
\centering
\includegraphics[width=.9\columnwidth]{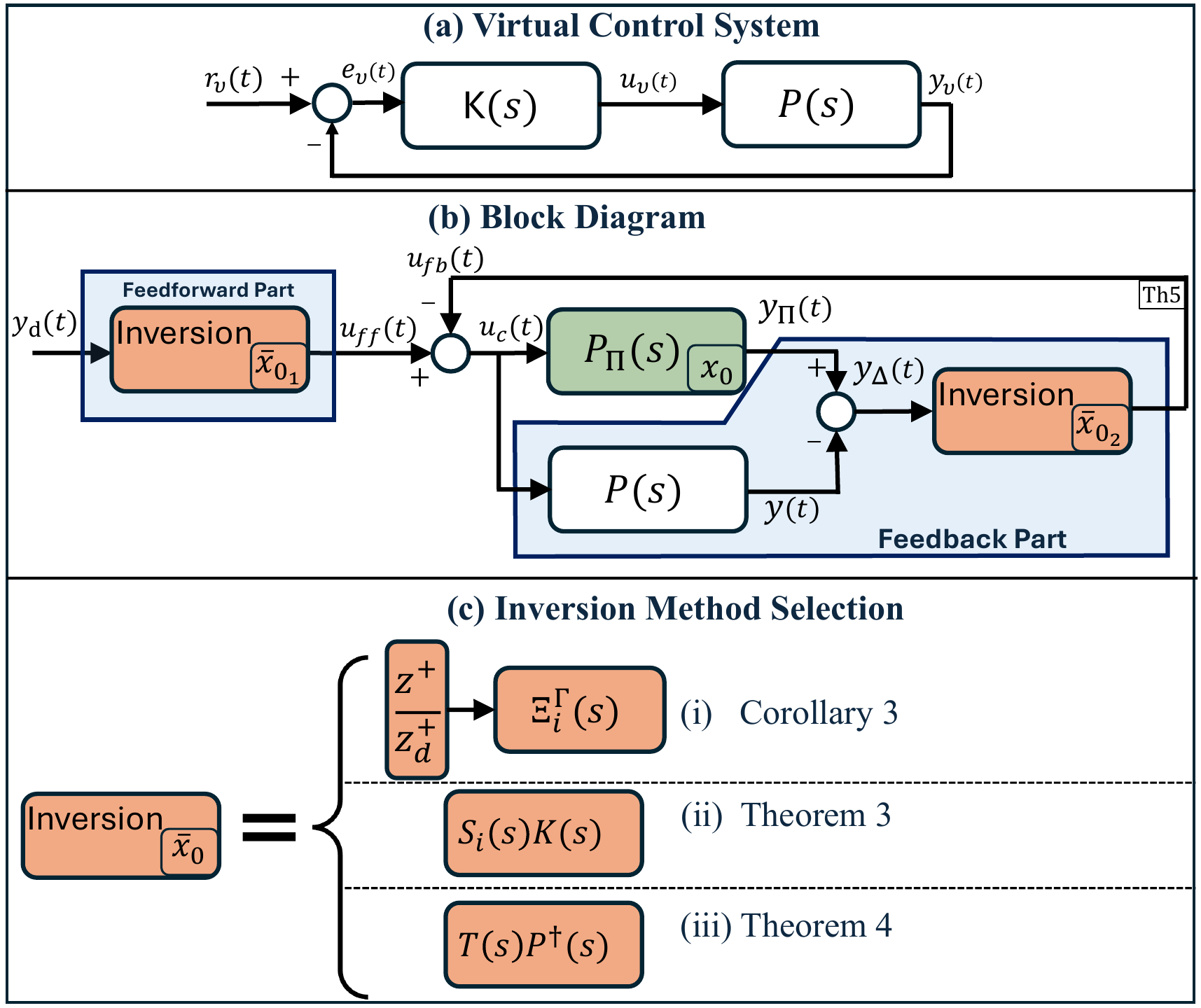}
\caption{Block diagram for closing the loop in stable inversion.}
\label{f:blockdiag}
\end{figure}
\begin{theorem} \label{th:feedback_unc}
Consider the scheme in Fig. \ref{f:blockdiag}.b. Let $ u_{ff}(t)=\Xi_a^\Gamma(y_d(t),\bar{x}_{0_1})$ and $u_{fb}(t)=\Xi_a^\Gamma(\bar{y}_\Delta(t), \bar{x}_{0_2})$ be designed over the $P(s)$ with $Z_R \ne \emptyset$. Suppose the actual system $P_\Pi$ (with $P_\Pi \ne P$ and unknown $x_0$) is as in \eqref{eq:phat}. Then, under $u_c(t) = u_{ff}(t) - u_{fb}(t)$ the tracking error satisfies
\begin{align} \label{eq:th6statement}
    \bigl( y_d(t)-{y}_\Pi(t) \bigr) \to \Laplace^{-1}\{ \normalfont 
 \text{res}[Y_\Pi]\} 
\end{align}
\end{theorem}
\begin{proof}
For brevity, let us denote $W_1 \Delta_u {W}_2$ simply by $\Delta_u$, and note that $u_c \triangleq u_c^{(f)}+u_c^{(0)} = u^{(f)}_{ff}+u^{(0)}_{ff}-u^{(f)}_{fb}-u^{(0)}_{fb}$. Although $\rank_{\mathbb{R}(s)}({P}_\Pi) \leq \rank_{\mathbb{R}(s)}({P})$, it may still occur that $\text{Im}_{\mathbb{R}(s)}(P_\Pi) \not \subseteq \text{Im}_{\mathbb{R}(s)}(P)$ implying $\text{res}[Y_\Pi] = \text{res}[Y_\Delta] \ne 0$. Without losing generality, we can re-define the bounds of integration as \cite[p.283-294]{debnath2016integral}
\begin{align}\label{eq:finitelaplace}
\Laplace \{ {y}_\Delta(t)\} = \int_0^\infty e^{-st}y_\Delta(t)dt = \int_0^{t-\epsilon} e^{-s\tau}y_\Delta(\tau)d\tau
\end{align}
where $\epsilon>0$ is chosen as small as to avoid the \textit{algebraic loop }issue.
Consider now the closed-loop configuration depicted in Fig. \ref{f:blockdiag}, as
\begin{align}
U_{fb}(s)-U_{fb}^{(0)}(s)&= \Xi_a^\Gamma(s){Y}_\Delta(s) = \Xi_i^\Gamma(s)\frac{z^+(s)}{z_d^+(s)}Y_\Delta(s)  \label{Ufb bary and y} \\
&=\frac{z^+(s)}{z_d^+(s)}\Xi_i^\Gamma(s)\bigr( \proj_{\im(P)}[Y_\Delta] + \res[Y_\Delta] \bigl) \nonumber
\end{align}
where $ \langle \res[Y_\Delta], q_k \rangle = 0$ implying $\Xi_i^\Gamma(s) \res[Y_\Delta] = 0$. Thus, the feedback only responds to the part of 
$Y_\Delta$ in $\im_{\mathbb{R}(s)}(P)$. Consequently,
\begin{align}
    PU_{fb}=\frac{z^+(s)}{z_d^+(s)} \proj_{\im(P)}[Y_\Delta]+PU^{(0)}_{fb}(s)
\end{align}
Next, by noting Remark~\ref{rem:no knowledge synthesis}, Theorem~\ref{th:nm_phase_inverses}-\ref{th:nm_phase_pseudoinverses}, and $P(s), \Delta_u \in \mathcal{RH}_\infty$, define 
\begin{align*}
D(s) \triangleq \bigl[ \Delta_uP(s) \bigl( U^{(0)}_{ff}(s) - U^{(0)}_{fb}(s)\bigr) + Y_\Pi^{(0)}(s) + U_{ff}^{(f)}(s) \bigr]
\end{align*}
where $Y_\Pi^{(0)}(s) = \Laplace^{-1}\{ C_\Pi e^{A_\Pi t}\,x_0 \}$ s.t. $ U^{(0)}_{ff}(s), U^{(f)}_{ff}(s), $ $U^{(0)}_{fb}(s), Y_\Pi^{(0)}(s)\in \mathcal{RH}_\infty$ are invariant under feedback, thus can be treated as a stable disturbance affecting the output. Then, the integral equation over $Y_\Delta$:
\begin{align*}
    &(I+\Delta_u)P(s)U_c(s) -P(s)U_c(s) + Y_\Pi^{(0)}(s) =\nonumber \\ &\Delta_uP(s)U_c(s)+ Y_\Pi^{(0)}(s) =Y_\Delta(s) = D(s)-\Delta_uP(s)U_{fb}^{(f)}(s).
\end{align*}
Using \eqref{Ufb bary and y} yields
\begin{align*}
    &D(s)-\frac{z^+(s)}{z_d^+(s)}\Delta_uP(s)\Xi_i^\Gamma(s)Y_\Delta(s) = Y_\Delta(s)\\
    &D(s) =\Bigl( I+ \frac{z^+(s)}{z_d^+(s)}\Delta_uP(s)\Xi_i^\Gamma(s) \Bigr)Y_\Delta(s)
\end{align*}
Since \small$P(I-V_{R1}\hat{V}_{R1})\bar{\kappa} = 0$,  $\max_\omega\bar{\sigma} \Bigl( P\hat{V}_R \begin{bmatrix}
    \Lambda^{-1}(s) & \textbf{0}\\
    \textbf{0} & \textbf{0}
\end{bmatrix}\hat{U}_R \Bigr)=1$\normalsize, $\biggl\Vert \frac{z^+(s)}{z_d^+(s)} \biggr\Vert_\infty =1$, and $\Vert \Delta_u\Vert_\infty <1$, it follows that:
\begin{align*}
\biggl\Vert \frac{z^+(s)}{z_d^+(s)}\Delta_uP(s)\Xi_i^\Gamma(s)\biggr\Vert_\infty= \alpha_b <1
\end{align*}
which ensures convergence of the following Neumann series.
\begin{align}
    \sum_{k=0}^\infty \biggl( \frac{z^+(s)}{z_d^+(s)}\Delta_uP(s)\Xi_i^\Gamma(s)\biggr)^k 
\end{align}
Thus, $\Bigl( I+ \frac{z^+(s)}{z_d^+(s)}\Delta_uP(s)\Xi_i^\Gamma(s)\Bigr)^{-1}$ exists and stable with
$\Vert Y_\Delta \Vert_\infty \leq {\Vert D(s)\Vert_\infty}/{\left( 1-\alpha_b\right) }$.
Then, the tracking error is
\begin{align*}
    Y_d(s)-Y_\Pi(s)&=Y_d(s)-Y_\Delta(s) - P(s)U_{c}(s). 
\end{align*}
With $\lim_{s \to 0}sD(s)=0$, Section~\ref{sec:almostiff} provides that $P(s)U_{ff}(s) \to Y_d(s)$ and thus
\begin{align*}
      Y_d(s)-Y_\Pi(s) &\to -Y_\Delta + P(s)U_{fb}(s) \\
   &= -\proj_{\im(P)}[Y_\Delta] - \res[Y_\Delta] + \frac{z^+(s)}{z_d^+(s)}\proj_{\im(P)}[Y_\Delta]
\end{align*}
Finally, 
$\mathrm{res}[Y_\Pi] = \mathrm{res}[Y_\Delta]$, and since \[\lim_{s \to 0} \biggl[ s \biggl(I - \frac{z^+(s)}{z_d^+(s)}\biggr) \proj_{\im(P)}[Y_\Delta] \biggr] = 0,\] standard final value arguments imply that as \(t \to \infty\), the components in \(\im_{\mathbb{R}(s)}(P)\) converges to zero. Hence,
\[
\lim_{t \to \infty} \bigl( y_d(t) - y_\Pi(t) \bigr) = \mathcal{L}^{-1}\bigl\{\mathrm{res}[Y_\Pi]\bigr\}.
\]
This shows that the tracking error converges to the inverse Laplace transform of the non-cancellable yet stable residual term, thus completing the proof.
\end{proof}
\begin{corollary} \label{coro:squareextunc}
Assume $P(s)$ is square, full rank implying $\res[Y_\Pi]=0$,  and $ Z_R \ne \emptyset$. Let $U_{ff}(s)=S_i(s)K(s)Y_d(s)$ and $U_{fb}(s)=S_i(s)K(s)\bar{Y}_\Delta(s)$. Then, $\bigl( y_d(t)-{y}_\Pi(t)\bigr) \to 0$.
\end{corollary}
\begin{corollary} \label{coro:relaxationsUNC}
For the depicted block diagram in Fig. \ref{f:blockdiag} with following conditions, \eqref{eq:th6statement} can also be re-written:
\begin{align}
\text{I. } \im_{\mathbb{R}(s)}&(P_\Pi) \subseteq \im_{\mathbb{R}(s)}(P) \text{ and } Z_{R} = \emptyset\  \nonumber\\
&\implies  \bigl( y_d(t)-y_\Pi(t) \bigr) = 0  \nonumber\\
\text{II. }\im_{\mathbb{R}(s)}&(P_\Pi) \subseteq \im_{\mathbb{R}(s)}(P) \text{ and } Z_{R} \ne \emptyset \nonumber\\ &\implies  \bigl( y_d(t)-y_\Pi(t) \bigr) \to 0 \nonumber \\
\text{III. }\im_{\mathbb{R}(s)}&(P_\Pi) \not\subseteq \im_{\mathbb{R}(s)}(P) \text{ and } Z_{R} = \emptyset \nonumber\\ &\implies  \bigl( y_d(t)-y_\Pi(t) \bigr) = \Laplace^{-1}\{ \normalfont 
 \text{res}[Y_\Pi]\}  \nonumber
\end{align}
\end{corollary}
Here, the term \( \Laplace^{-1}\{\text{res}[.]\} \) corresponds directly to \(\varphi\) as described in \eqref{ESwithinevitabeerrors}, representing the contribution of inevitable errors. In this paper, time delays are not treated as part of the uncertainty; however, the framework can be extended to handle stochastic and delayed systems by leveraging mean-square exponential stability techniques \cite{9410357}.
On the other hand, for implementation, the compactness of the algebraic structures allows for facilitating straightforward solutions. However, when utilizing $\mathcal{H}_\infty$-based approximate solutions, we encounter complex, high but finite-order structures where using the balanced model reduction is a solution.
\section{Numerical Examples}\label{sec:examples}
In this section, we present some numerical examples to illustrate the effectiveness of the proposed approach. 
\begin{example} \label{sec:ex2}
Consider, as in \cite{saeki2023model}, the following $2 \times 2$, full rank, and minimal system with $Z_P=\{-10,-0.86,1\}$, $Z_R={1}$.
\begin{align} \label{ex:2}
    P(s)=\begin{bmatrix}
        \frac{(1-s)}{(s+1)^2} & \frac{0.3}{s+0.5} \\
        \frac{-(1-s)}{(s+1)^2(s+2)} & \frac{2}{s+3}
    \end{bmatrix}
\end{align}
In this case, $\rank_{\mathbb{R}(s)}(P) = 2 = n_u = n_y$, which implies that $\res[Y] = 0$. For the desired loop shape $\text{diag}\left[\frac{5}{s}, \frac{5}{s}\right]$, we solve \eqref{eq:min_maxSV} with $\gamma = 3.6$. The approximate inverse $\Xi_a^\Gamma(s)=S_i(s)K(s)$ is plotted in Fig.~\ref{f:ex2}\,(right), in line with Theorem~\ref{th:nm_phase_inverses}.  
The control signal $u_{\text{inv}}(t)$ appears in Fig.~\ref{f:ex2}\,(middle), and the output response in Fig.~\ref{f:ex2}\,(left).  
Relative to the inner–outer-factorization benchmark \cite[Fig.~8, black curves]{saeki2023model}, our feed-forward design achieves markedly better tracking.
\begin{figure}
\centering
\includegraphics[width=0.9\columnwidth,height=0.095\textheight]{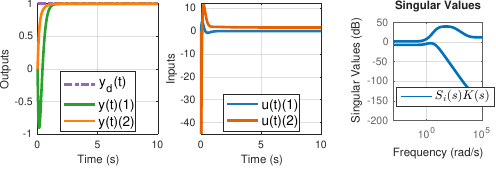}
\caption{Output tracking (left), control input (middle), and singular-value plot of the approximate inverse $S_i(s)K(s)$ (right).}
\label{f:ex2}
\end{figure}
\end{example}
Now, consider that the real system $P_\Pi(s)$ differs from $P(s)$:
\begin{align} \label{ex:2}
    P_\Pi(s) = \begin{bmatrix}
        \frac{(0.2-s)}{s^2 + 20s + 100} & \frac{0.3}{s + 0.1} \\
       \frac{s^2 + 1.8s - 0.4}{s^2 + 8s + 16}  & \frac{2}{s + 2.1}
    \end{bmatrix},
\end{align}
where $\text{Z}_{P_\Pi} = \{-20.4, -6.4, -2.7, -1.2, 0.2\} \neq \text{Z}_P$ and $\text{Z}_{R_\Pi} = \{0.2\} \neq \text{Z}_R$. Moreover, $P_\Pi(s)$ has poles at $s_{1,2,3,4,5,6} = -10, -10, -4, -4, -0.1, -2.1$, which are different from the poles of the nominal system $P(s)$, $s_{1,2,3,4} = -1, -1, -0.5, -3$. According to Theorem~\ref{th:feedback_unc}, or more specifically, Corollary~\ref{coro:relaxationsUNC}.II, to handle the uncertainty, we employ Fig. \ref{f:blockdiag}.(b) with $S_i(s)K(s)$ by letting $\epsilon_i\to0$ for both inversion blocks.
Two simulations were performed. In both cases, a unit step disturbance is applied to the output at \(t=125\)~s. In the first, zero initial states are used for both the inverse and forward systems, while in the second, random initial states are assigned to illustrate the effectiveness of the proposed approach: \footnotesize $x_0 = [3.1, 4.1, -3.7,4.1,1.3,-4], \bar{x}_{0_1}=[-0.2,0.6,1.9,1.9,-0.5,1.9,1.8,0.4,1.4,-0.6,0.3,1.7,1.4,1.8], \bar{x}_{0_2} =[0.3,-0.9,0.7,0.9,0.4,0.5,0.5,0.5,-0.2,0.3,-0.7,0.4,-1,-0.4,-0.9].$ \normalsize The reference tracking performance shown by Fig. \ref{f:ex4_outs} with a bounded input validating the theory. 
\begin{figure}
\centering
\includegraphics[width=0.8\columnwidth]{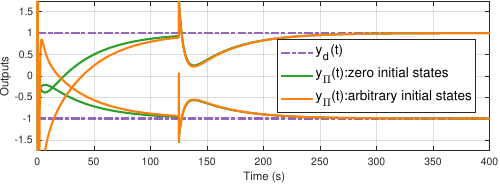}
\caption{The output tracking of Example 1 under uncertainty with zero and arbitrary initial states.}
\label{f:ex4_outs}
\end{figure}
\begin{example} \label{sec:ex3}
Now, consider the following vectorial system \small
\begin{align} \label{ex:3}
    P(s)=\begin{bmatrix}
        \frac{s^2-5s-50}{s^2+3s+2} &
        \frac{s-10}{s^2+3s+2} 
    \end{bmatrix}^T
\end{align}\normalsize
where it has an invariant zero at $s=10$, thus it is a non-minimum phase and underactuated system. Define $Y_d$ as
\begin{align}
    Y_d(s)=\begin{bmatrix}
    \frac{2}{s^2+1}+\frac{8}{s} & \frac{1}{s}
    \end{bmatrix}^T \notin \im_{\mathbb{R}(s)}(P) \text{  s.t. } \res[Y_d] \ne 0.
\end{align}
Then, $Y(s)=P(s)\Xi_a^\Gamma(s)Y_d(s)$ such that $\res[Y] \ne 0 \impliedby Y_d(s) \notin  \im_{\mathbb{R}(s)}(P) \cap \mathcal{RH}_\infty$. Based on Theorem~\ref{th:nm_phase_pseudoinverses}, define $P_p(s)=\frac{s^2-5s-50}{s^2+3s+2}$ which also has an invariant zero at $s=10$. Solving \eqref{eq:min_maxSV} yields a stable $T_p(s)$ with $\gamma = 2.32$. Also, based on \eqref{ex:3}, we have
\begin{align}
    P^\dagger(s)=\begin{bmatrix}\frac{s^3+8s^2+17s+10}{s^3-74s-260} & \frac{s^2+3s+2}{s^3-74s-260} \end{bmatrix}
\end{align}
yielding $T_p(s)P^\dagger(s) \in \mathcal{RH}_\infty$. The reference tracking performance of the form in Theorem \ref{th:nm_phase_pseudoinverses}, the design of $T_p(s)$, and the boundedness of $u(t)$ are shown by Fig. \ref{f:ex3_outs}. As it can be seen from Fig. \ref{f:ex3_outs}, the errors are fully in harmony with Definition~\ref{def:projres}.
\begin{figure}
\centering
\includegraphics[width=0.9\columnwidth,height=0.1\textheight]{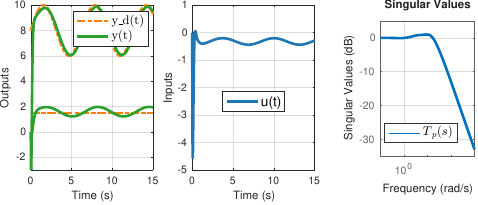}
\caption{Output tracking (left), control input (middle), and singular value plot of the approximate inverse \(T_p(s)\) (right).}
\label{f:ex3_outs}
\end{figure}
\end{example}
\section{Conclusion} \label{sec:conc}
We presented a unified algebraic framework for stable inversion, covering non-minimum-phase, nonsquare, and singular MIMO systems.  Necessary and sufficient conditions were established, and constructive inverses were shown with exponential error decay without preview. Uncertainty is handled by an orthogonal projection that isolates the reachable portion of the (desired) output, with the residual captured by the irreducible term~\(\varphi\).  A feed-forward/feedback loop then stabilizes the system under uncertainties.  Future work will develop data-driven techniques to identify the system’s reachable output subspace and adapt the inversion scheme accordingly.

\ifCLASSOPTIONcaptionsoff
  \newpage
\fi


\bibliographystyle{IEEEtran}
\bibliography{stab_inv}
\end{document}